\documentclass[11pt]{amsart}
\usepackage{tikz}
\usepackage{times}
\usepackage{color}
\usepackage{amsmath}
\usepackage{amssymb}
\usepackage{amsthm}
\usepackage{enumerate}
\usepackage{amsbsy}
\usepackage{amsfonts}
\newtheorem{theo}{Theorem}[section]
\newtheorem{defin}[theo]{Definition}
\newtheorem{prop}[theo]{Proposition}

\newtheorem{lemm}[theo]{Lemma}
\newtheorem{rem}[theo]{Remark}

\newcommand{\al}{\alpha}

\newcommand{\Ga}{\Gamma}

\newcommand{\om}{\omega}

\newcommand{\ep}{\epsilon }

\newcommand{\De}{\Delta}
\newcommand{\de}{\delta}

\newcommand{\pa}{\partial}

\newcommand{\R}{{\mathbb R}^n}
\newcommand{\hR}{{\mathbb R}^n_+}

\newcommand{\ri}{\rightarrow}
\newcommand{\Rn}{{\mathbb R}^{n-1}}

\newcommand{\na}{\nabla}

\newcommand{\bke}[1]{\left( #1 \right)}

\newcommand{\bket}[1]{\left\{ #1 \right\}}
\newcommand{\norm}[1]{\left\Vert #1 \right\Vert}
\newcommand{\abs}[1]{\left| #1 \right|}
\newcommand{\rabs}[1]{\left. #1 \right|}

\begin{document}
\baselineskip=18pt

\title[]{On Caccioppoli's inequalities of Stokes equations and Navier-Stokes equations near boundary}
\author{Tongkeun Chang and Kyungkeun Kang}

\thanks{}

\begin{abstract}

We study Caccioppoli's inequalities of the non-stationary Stokes
equations and Navier-Stokes equations. Our analysis is local near
boundary and we prove that, in contrast to the interior case, the
Caccioppoli's  inequalities of the Stokes equations and the
Navier-Stokes equations, in general, fail near boundary.

\end{abstract}

\maketitle

\section{Introduction}
\setcounter{equation}{0}

We consider first non-stationary Stokes equations near flat boundary
\begin{equation}\label{Stokes-10}
w_t - \De w + \na \pi =0\qquad \mbox{div } w =0 \quad \mbox{ in }
\,Q^+_{(0,1),1}:=B^+_{0,1}\times (0, 1),
\end{equation}
where $B^+_{0,r}:=\{ x=(x', x_n )\in \R: |x|<r, x_n >0\}$ for $r>0$. Here
no-slip boundary condition is given on the flat boundary, i.e.
\begin{equation}\label{Stokes-20}
w=0 \quad \mbox{ on } \,\Sigma:=(B_{0,1}\cap\{x_n=0\})\times (0,1),
\end{equation}
where $B_{0,r}:=\{ x\in \R: |x|<r \}$ for $r>0$. We emphasize that boundary
conditions are prescribed only on flat boundary of $Q^+_{(0,1),1}$, not on the
rounded boundary, $\{ x\in \R: |x|=1, x_n> 0 \}\times (0,1)$. From
now on, we denote $Q^+_{(0,1),1}$ and $B^+_{0,1}$ by $Q^+_1$ and
$B^+_1$, respectively, unless any confusion is to be expected.

We can compare similar situation to the heat equation, i.e.
\begin{align*}
v_t - \De v =0\quad \mbox{ in } \,Q^+_1
\end{align*}
with homogeneous boundary condition
\begin{align*}
v=0 \quad \mbox{ on } \,\Sigma.
\end{align*}
It is then well-known that the following  priori estimate, so called
Caccioppoli's inequality, is available:
\begin{equation}\label{Heat-50}
\norm{\nabla v}_{L^2(Q^+_{\frac{1}{2}})}\le c\norm{v}_{L^2(Q^+_1)},
\end{equation}
where $Q^+_{\frac{1}{2}}:=B^+_{\frac{1}{2}}\times (\frac{3}{4}, 1)$
and $c$ is independent of $v$. For the Stokes equations
\eqref{Stokes-10}-\eqref{Stokes-20}, an energy estimate shows
\begin{equation}\label{Stokes-55}
\norm{\nabla w}_{L^2(Q^+_{\frac{1}{2}})}\le
c\bke{\norm{w}_{L^2(Q^+_1)}+\norm{\pi}_{L^2(Q^+_1)}},
\end{equation}
and, to the authors' knowledge, the Caccioppoli's inequality as in
\eqref{Heat-50} is not known for the Stokes equations
\eqref{Stokes-10}-\eqref{Stokes-20}, namely it is unknown whether or
not the following inequality is available;
\begin{equation}\label{Stokes-57}
\norm{\nabla w}_{L^2(Q^+_{\frac{1}{2}})}\le c\norm{w}_{L^2(Q^+_1)},
\end{equation}
where $c$ is independent of $w$.

We remark that it was shown in \cite{Kang05} that the maximum of
normal derivatives of tangential components for Stokes equations are
not controlled by the righthand side of \eqref{Stokes-57}. More
precisely, an example of the Stokes equations
\eqref{Stokes-10}-\eqref{Stokes-20} was constructed such that
$\displaystyle\sup_{Q^+_{1/2}}|D_{x_3} w|$ is not bounded, although
$\int_{Q^+_1}|w|^2$ is finite. Furthermore,  solutions of the Stokes
equations \eqref{Stokes-10}-\eqref{Stokes-20} may not be even H\"older
continuous, unless corresponding pressures are  integrable (see \cite[Remark
6]{Kang05} and compare to \cite{Seregin00}).

We can also consider the Navier-Stokes equations near flat boundary,
i.e.
\begin{align*}
u_t - \De u + \na p =-\mbox{div}(u\otimes u), \qquad \mbox{div } u
=0 \quad \mbox{ in } \,Q^+_1
\end{align*}
with no-slip boundary condition
\begin{align*}
u=0 \quad \mbox{ on } \,\Sigma.
\end{align*}
Again, we emphasize that homogeneous boundary conditions are
assigned only on flat boundary of $Q_1^+$. We can ask whether or not the
following Caccioppoli type's inequality of the Navier-Stokes
equations is satisfied:
\begin{equation}\label{NSE-80}
\norm{\nabla u}^2_{L^2(Q^+_{\frac{1}{2}})}\le
c\bke{\norm{u}^2_{L^2(Q^+_1)}+\norm{u}^3_{L^3(Q^+_1)}},
\end{equation}
where $c$ is independent of $u$. The main point of the above
inequality is that pressure does not appear in the righthand side.

Our main goal is to show that it is not, in general, true to obtain
near boundary the Caccioppoli's inequality \eqref{Stokes-57} of the
Stokes equations and Caccioppoli type's inequality \eqref{NSE-80} of
the Navier-Stokes equations. More precisely, we construct  sequences
of smooth solutions (for example $W^{2,1}_2 (Q^+_1)$-solutions) of the Stokes equations and the Navier-Stokes equations
such that righthand-sides of \eqref{Stokes-57} and \eqref{NSE-80}
are uniformly bounded but left-hand sides are not uniformly bounded, i.e.
$L^2-$norms of gradient of velocities, say $\norm{\nabla u_n}^2_{L^2(Q^+_{\frac{1}{2}})}$, tend to infinity
by passing them to the limit (see Theorem
\ref{maintheorem}).

For the interior case, it was independently shown in  \cite{CSYT08},
\cite{J} and \cite{Wolf15} by different method of proofs that the
Caccioppoli's inequality \eqref{Stokes-57} of the Stokes equations is
valid in the interior.

The Caccioppoli type's inequality \eqref{NSE-80} was also extended in the interior
to {\it suitable} weak solutions of some nonlinear fluid equations,
e.g. the Navier-Stokes equations \cite{Wolf15}, Magnetohydrodynamics
equations \cite{Choe-Yang15} and non-Newtonian fluid flow
\cite{Jin-Kang17}.

We remark that the Caccioppoli type's inequality of the stationary
case was obtained in the interior or near boundary (refer to e.g.
\cite{Giaquinta-Modica}, \cite{Kang04} and \cite{Sverak-Tsai98}).

Our main result reads as follows:
\begin{theo}\label{maintheorem}
The Caccioppoli's inequality \eqref{Stokes-57} of the Stokes equations
and Caccioppoli type's inequality \eqref{NSE-80} of the
Navier-Stokes equations, in general, fail  near boundary.
\end{theo}

We remark that our analysis is only local near boundary. Our
construction of sequences of solutions failing \eqref{Stokes-57} or
\eqref{NSE-80} are caused by non-local behavior of solutions and
singular boundary conditions away from flat boundary. Therefore,
such construction would not be applicable to the Stokes and the
Navier-Stokes equations in domains with no-slip conditions
on boundaries everywhere.

We briefly give the mainstream of
how we show Theorem \ref{maintheorem}.

Firstly, we construct, in a half-space, a very weak solution of the Stokes equations
whose normal derivatives of tangential components is not locally square integrable near boundary.
More precisely, we use the explicit formula \eqref{rep-bvp-stokes-w} of the Stokes equations with a
boundary condition given in \eqref{boundarydata} and \eqref{0502-6}.
In particular, among all split terms of the normal derivative of tangential component, it turns out that the following integral (see \eqref{0411-v}) is not square integrable in a local neighborhood near boundary (see Proposition \ref{lemma0406}).
\[
D_{x_n} \int_0^t \int_{\Rn} B_{in}(x'- y', x_n, t -s) g_n(y', s) dy'ds,\qquad i=1,2, \cdots n-1,
\]
where $B_{in}$ is given in \eqref{B-tensor}.
Indeed, one crucial estimate
is
\[
\int_{\Rn} e^{-\frac{|X'-z'|^2}{t-s}} \frac{z_1}{|z'|^n} dz' \gtrsim (t-s)^{\frac{n-1}2}\bke{1-e^{-\frac{c}{t-s}}}, \quad |X'| \geq \frac12,
\]
which causes the unbounded $L^2$-norm of normal derivative in $B_r(0)\times (t_0, t_0+{r^2})$
(see \eqref{est-J1J2J3}, \eqref{grad-w1} and \eqref{L2-blowup}). On the other hand, we can see that
the solution is integrable in $L^4 (0, \infty; L^p(\R_+))$ for $n/(n-1)<p<\infty$ (see Proposition \ref{lemma0406}).

Next step is to regularize the boundary data that is originally singular so that corresponding solutions of the Stokes equations
are regular. If the Caccioppoli's inequality near boundary is valid, then the regular solutions
should satisfy the inequality. If this is the case, by passing to the limit, the gradient of the limit solution must
be square integrable near boundary, which leads to a contradiction (see for details Subsection 4.1 in Section 4).

For the Navier-Stokes equations, we consider the similar situation as in the Stokes equations.
If we denote by $w$ the solution of the Stokes equations mentioned above, we look for a solution $u$
of the form $u=w+v$ such that $v$ solving the following perturbed Navier-Stokes equations:
\[
v_t-\Delta v+\nabla\pi=-{\rm div}\,( (v+w)\otimes (v+w)), \qquad {\rm div}\, v=0
\]
with zero boundary and zero initial data.
In fact, we construct a weak solution $v$ whose gradient is square integrable, in case that the size of
$w$ is assumed to be a sufficiently small.
Therefore, similar argument as in Stokes equations yields that the Caccioppoli's inequality near boundary
fails in general for the Navier-Stokes equations as well (see for details Subsection 4.2 in Section 4).

This paper is organized as follows. In Section 2, we recall some
known results and introduce new estimates that are useful for our purpose. Section 3 is devoted
to constructing a very weak solution in a half space such that
$L^2-$norm of its gradient is not bounded. In Section 4, we present
the proof of Theorem \ref{maintheorem}.

\section{Preliminaries}
\setcounter{equation}{0}

For notations,  we denote $x=(x^{\prime},x_n)$, where the symbol
${\prime}$ means the coordinate up to $n-1$, that is,
$x^{\prime}=(x_1,x_2,\cdot\cdot\cdot,x_{n-1})$. We write $D_{x_i} u$
  as the partial derivative of $u$ with respect to $x_i, \,\, 1 \leq i \leq n$, i.e., $D_{x_i} u(x) =
\frac{\pa }{\pa x_i}u(x)$. Throughout this paper we denote by $c$ various generic positive constant and by $c(*,\cdots,*)$ depending on the quantities appearing in the
parenthesis.

Next, we introduce notions of {\it very weak solutions} for the
Stokes equations and the Navier-Stokes equations with non-zero boundary
values in a half-space $\hR$. To be more precise, we consider first
the following Stokes equations in $\hR$ with non-zero boundary values:
\begin{equation}\label{Stokes-bvp-200}
w_t - \De w + \na \pi = {\rm div} \,F, \qquad \mbox{div } w =0,\quad
\mbox{ in }
 \hR \times (0,\infty)
\end{equation}
with
\begin{equation}\label{Stokes-bvp-210}
\rabs{w}_{t=0}= 0, \qquad  \rabs{w}_{x_n =0} = g.
\end{equation}
We now define very weak solutions of
\eqref{Stokes-bvp-200}-\eqref{Stokes-bvp-210}.
\begin{defin}\label{stokesdefinition}
Let $g \in L^1_{\rm{loc}} (\pa \hR\times (0, T))$ and $ F \in
L^1_{\rm{loc}}(\hR\times (0, T))$. A vector field $u\in
L^1_{\rm{loc}}( \hR \times (0, T))$ is called a very weak solution
of the Stokes equations \eqref{Stokes-bvp-200}-\eqref{Stokes-bvp-210},
if the following equality is  satisfied:
\begin{align*}
-\int^T_0\int_{\hR }u\cdot \Delta \Phi
dxdt=\int^T_0\int_{\hR}\bke{u\cdot \Phi_t-F:\nabla \Phi} dxdt
-\int^T_0 \int_{\pa \hR} g \cdot D_{x_n}\Phi  dx' dt
\end{align*}
for each $\Phi\in C^2_c(\overline{\hR}\times [0,T])$ with
\begin{align}\label{0528-1}
\mbox{\rm div } \Phi=0,\quad \Phi\big|_{\pa \hR \times (0, T)}=0, \quad
\Phi(\cdot,T) =0.
\end{align}
In addition, for each $\Psi\in C^1_c(\overline{\hR})$ with
 $\Psi \big|_{\pa \hR }=0$
\begin{equation}\label{Stokes-bvp-2200}
\int_{\hR} u(x,t) \cdot \na \Psi(x) dx =0  \quad  \mbox{ for all}
\quad 0 < t< T.
\end{equation}
\end{defin}

We recall  existence results of the Stokes equations
\eqref{Stokes-bvp-200}-\eqref{Stokes-bvp-210}.
\begin{prop}\label{thm-stokes}(\cite[Theorem 1.2]{CJ})\,\,
Let  $1<p, q<\infty$. Assume that $g\in
L^q(0,\infty;\dot{B}^{-\frac{1}{p}}_{pp}(\Rn))$ and $F\in
L^{q_1}(0,\infty; L^{p_1}(\hR)) $, where $p_1, q_1$ satisfy
$n/p_1+2/q_1=n/p+2/q+1$ with $q_1\leq q$ and $p_1 \leq p$. Then, there is the unique
very weak solution $w\in L^q(0,\infty;L^p(\R_+))$ of the Stokes equations
\eqref{Stokes-bvp-200}-\eqref{Stokes-bvp-210} such that $w$
satisfies
\[
\| w\|_{L^q(0,\infty;L^p(\R_+))}\leq c \bke{\|F\|_{L^{q_1}(0,\infty; L^{p_1}(\R_+))} +\|
g\|_{L^q(0,\infty;\dot{B}^{-\frac{1}{p}}_{pp}(\Rn))}}.
\]
\end{prop}
Here, $\dot{B}^{\al}_{pp}(\Rn)$ is homogeneous Besov space in $\Rn$ (see \cite{BL} for definition and properties of homogeneous Besov spaces).

The estimate of the following proposition may be known to the experts, but we couldn't find it in the literature.
Since it will be used to prove Theorem \ref{maintheorem}, we give its details in Appendix \ref{appendixa}.
\begin{prop}\label{theo0503}
Let $1< p, q< \infty$, $F \in L^q (0,\infty; L^p (\R_+))$, $\rabs{F}_{x_n =0} \in L^q (0, \infty; \dot B^{-\frac1p}_{pp} (\Rn))$ and  $g =0$. Then, there is the unique
very weak solution $w $ of the Stokes equations
\eqref{Stokes-bvp-200}-\eqref{Stokes-bvp-210} such that $w$
satisfies
\begin{align}\label{0504-1}
\|  \na w\|_{L^q (0, \infty;   L^p (\R_+)} \leq c \bke{\| F\|_{L^{q} (0,\infty;  L^p (\R_+))} +  \| \rabs{F}_{x_n =0}\|_{L^q (0, \infty; \dot B^{-\frac1p}_{pp}(\Rn))}}.
\end{align}
\end{prop}
In \cite{GGS}, the authors showed \eqref{0504-1} with additional condition  ${\rm div}\, F =0$ and $F_{in}|_{x_n =0}=0, \, i =1, \cdots , n$. In \cite{KS}, the authors    showed \eqref{0504-1} with condition $p =q$, then the second term of the right-hand side is dropped.

We also consider the boundary value problem of the
Navier-Stokes equations in a half-space, namely
\begin{equation}\label{NSE-bvp-230}
u_t - \De u + \na p =-\mbox{div}(u\otimes u), \qquad \mbox{div } u
=0 \qquad\mbox{ in } \hR \times (0,\infty)
\end{equation}
with
\begin{equation}\label{NSE-bvp-240}
\rabs{u}_{t=0}= 0, \qquad  \rabs{u}_{x_n =0} = g.
\end{equation}
We mean by very weak solutions of the Navier-Stokes equations
\eqref{NSE-bvp-230}-\eqref{NSE-bvp-240} as follows:

\begin{defin}
Let $g \in L^1_{\rm{loc}}(\pa \hR\times (0, T))$. A vector field
$u\in L^2_{loc}( \hR \times (0, T))$ is called a very weak solution
of the non-stationary Navier-Stokes equations
\eqref{NSE-bvp-230}-\eqref{NSE-bvp-240}, if the following equality
is  satisfied:
\begin{align*}
-\int^T_0\int_{\hR}u\cdot \Delta \Phi dxdt=\int^T_0\int_{\hR}(u\cdot
\Phi_t+(u\otimes u):\nabla \Phi) dxdt -\int^T_0 \int_{\pa \hR} g
\cdot D_{x_n} \Phi dx' dt
\end{align*}
for each $\Phi\in C^2_c(\overline{\hR}\times [0,T])$ satisfying \eqref{0528-1}.
In addition, for each $\Psi\in C^1_c(\overline{\hR})$ with
 $\Psi \big|_{\pa \hR }=0$,  $u$ satisfies \eqref{Stokes-bvp-2200}.
\end{defin}


\section{Stokes equations with boundary data in a half-space}
\setcounter{equation}{0}

 We let $N$ and $\Gamma$  be the fundamental solutions to the Laplace equation
 and the heat equation, respectively, i.e.
\begin{align*}
N(x) = \left\{\begin{array}{ll}\vspace{2mm} -\frac{1}{(n-2)\om_n} \frac{1}{|x|^{n-2}},& n \geq 3\\
\frac1{2\pi} \ln |x|, & n =2
\end{array}
\right. \qquad
\Gamma(x,t)  = \left\{\begin{array}{ll} \vspace{2mm}
\frac{1}{\sqrt{4\pi t}^n} e^{ -\frac{|x|^2}{4t} }, &   t > 0,\\
0, & t < 0,
\end{array}
\right.
\end{align*}
where $\om_n$ is the measure of the unit sphere in $\R$. For
convenience, we introduce a tensor $L_{ij}$ and a scalar function
$A$ defined by
\begin{align}\label{L-tensor}
{L}_{ij} (x,t) & =  D_{x_j}\int_0^{x_n} \int_{\Rn}   D_{z_n}
\Ga(z,t) D_{x_i}   N(x-z)  dz,\quad
i,j=1,2,\cdots,n,\\
\notag A(x,t) & =\int_{\Rn}\Ga(z^{\prime},0,t)N(x^{\prime}-z^{\prime},x_n)dz^{\prime}.
\end{align}

The Poisson kernel $(K,\pi ) $ of the Stokes equations is  given as
\[
K_{ij}(x'-y',x_n,t) =  -2 \delta_{ij} D_{x_n}\Ga(x'-y', x_n,t)
-L_{ij} (x'-y',x_n,t)
\]
\begin{equation}\label{Poisson-tensor-K}
+  \de_{jn} \de(t)  D_{x_i} N(x'-y',x_n),
\end{equation}
\[
\pi_j (x'-y',x_n,t )=-2\delta(t) D_{x_j}D_{x_n}N(x'-y',x_n)+4
D_{x_n}D_{x_n}A(x'-y',x_n,t)
\]
\begin{equation}\label{Poisson-tensor-P}
+4D_{t} D_{x_j}A(x'-y',x_n,t),
\end{equation}
where $\delta(t)$ is the Dirac delta function and $\delta_{ij}$ is
the Kronecker delta function.

We recall the following relations on $L$  (see \cite{So}):
\begin{align} \label{1006-3}
\sum_{i=1}^{n} L_{ii} = -2D_{x_n} \Ga, \quad\qquad  L_{in} =
L_{ni}  + B_{in} \,\, \mbox{ if }\, i \neq n,
\end{align}
where
\begin{equation}\label{B-tensor}
B_{in}(x,t) = -\int_{\Rn}D_{x_n} \Ga(x^\prime -z^\prime ,
x_n, t) D_{z_i} N( z^{\prime},0) dz^\prime.
\end{equation}
Furthermore, we remind estimates of $L_{ij}$ defined in
\eqref{L-tensor} (see \cite{So}).
\begin{equation}\label{est-L-tensor}
|D^{l_0}_{x_n} D^{k_0}_{x'} D_{t}^{m_0} L_{ij}(x,t)| \leq
\frac{c}{t^{m_0 + \frac12} (|x|^2 +t )^{\frac12 n + \frac12 k_0}
(x_n^2 +t)^{\frac12 l_0}},
\end{equation}
where $ 1 \leq  i \leq n$ and $1 \leq j \leq n-1$.

It is shown in \cite{So} that the solution $(w,\pi)$ of the Stokes
equations \eqref{Stokes-bvp-200}-\eqref{Stokes-bvp-210} with $ F =0$ is expressed by
\begin{align}\label{rep-bvp-stokes-w}
w_i(x,t) = \sum_{j=1}^{n}\int_0^t \int_{\Rn} K_{ij}(
x^{\prime}-y^{\prime},x_n,t-s)g_j(y^{\prime},s) dy^{\prime}ds,\\
\label{rep-bvp-stokes-p}
\pi(x,t) = \sum_{j=1}^{n}\int_0^t \int_{\Rn}
\pi_j(x^{\prime}-y^{\prime},x_n,t-s) g_j(y^{\prime},s)
dy^{\prime}ds,
\end{align}
where $(K_{ij}, \pi_j)$ is Poisson kernel of the Stokes equations given
in \eqref{Poisson-tensor-K} and \eqref{Poisson-tensor-P}.

Next, we will construct a solution $w$ of Stokes equations via
\eqref{rep-bvp-stokes-w} and \eqref{rep-bvp-stokes-p} for a certain $g$ such that
$L^2$-norm of $\nabla w$ is not bounded. For convenience, we denote
\[
A = \{ y'=(y_1, y_2, \cdots, y_{n-1})\in \Rn \, | \, 1 < |y'| < 2, \,  -2 < y_i
< -1, \, \, 1 \leq i \leq n-1 \}.
\]
We note that if $y' \in A$, then $|y'| \leq  2< -2y_i$ for all $1\leq i\leq n-1$. Let
\begin{align}\label{boundarydata}
g^1_n(y') = \chi_{A}(y'),\qquad  g^2_n(s) =
\bke{s-\frac1{16}}^{-\frac14} \abs{\ln \bke{s-\frac1{16}}}^{-\frac14
-\ep}\chi_{\frac1{16} < s< \frac14}(s),
\end{align}
where $\ep$ is a fixed constant with $0 <\ep \leq \frac14$ and $\chi$ is characteristic function. We introduce
a non-zero boundary data $g:\mathbb R^{n-1}\times \mathbb
R_+\rightarrow  \R$ with only $n-$th component defined
by
\begin{align}\label{0502-6}
g(y', s) = (0, \cdots, 0,  g_n (y',s))= (0, \cdots, 0,\al g^1_n (y')g^2_n(s)),
\end{align}
where $\al> 0$ is determined late on.

The following is two dimensional cartoon for $A$ that is the support of $g_n^1$.


\begin{tikzpicture}[scale=3]


\draw (0,0) ellipse (8.5pt and 8.5pt);

\draw (0,0) ellipse (17pt and 17pt);

\draw (0,0) ellipse (34pt and 34pt);

\draw (-0.8, -0.75) node {$A$};

\draw(0.13,1.5) node {$y_2$};

\draw (0,1.5) --(-0.05, 1.45);
\draw (0,1.5) --(0.05, 1.45);


\draw (-2,0) -- (2,0);

\draw (-2,-1.2) -- (2,-1.2);

\draw (-2,-0.6) -- (2,-0.6);

\draw (0,-1.5) -- (0, 1.5);

 \draw (-0.6, 1.5) -- (-0.6,-1.5);

\draw (-1.2, 1.5) -- (-1.2,-1.5);

\draw (-0.65, -0.7) node {$ \bullet$};

\draw (-0.55, -0.75) node {$  y'$};

\draw (-0.32, -0.1) node {$-\frac12$};

\draw (-0.62, -0.1) node {$-1$};

\draw (-1.22, -0.1) node {$-2$};

\draw (2.05,-0.099) node {$y_1$};

\draw (2,0) --(1.95, 0.05);
\draw (2,0) --(1.95, -0.05);

\draw (-0.2, -0.15) node {$ \bullet$};
\draw (0.0, 0.0) node {$ \bullet$};
\draw (0.05, -0.05) node {$ 0$};

\draw (-0.1,-0.1 ) node {$x'$};

\draw (-0.68,-0.98) -- (-0.6,-0.9);

\draw (-0.743,-0.945) -- (-0.6,-0.8);

\draw (-0.799,-0.899) -- (-0.6,-0.7);

\draw (-0.85,-0.85) -- (-0.6,-0.6);

\draw (-0.9,-0.8) -- (-0.7,-0.6);

\draw (-0.95,-0.74) -- (-0.8,-0.6);

\draw (-0.98,-0.68) -- (-0.9,-0.6);

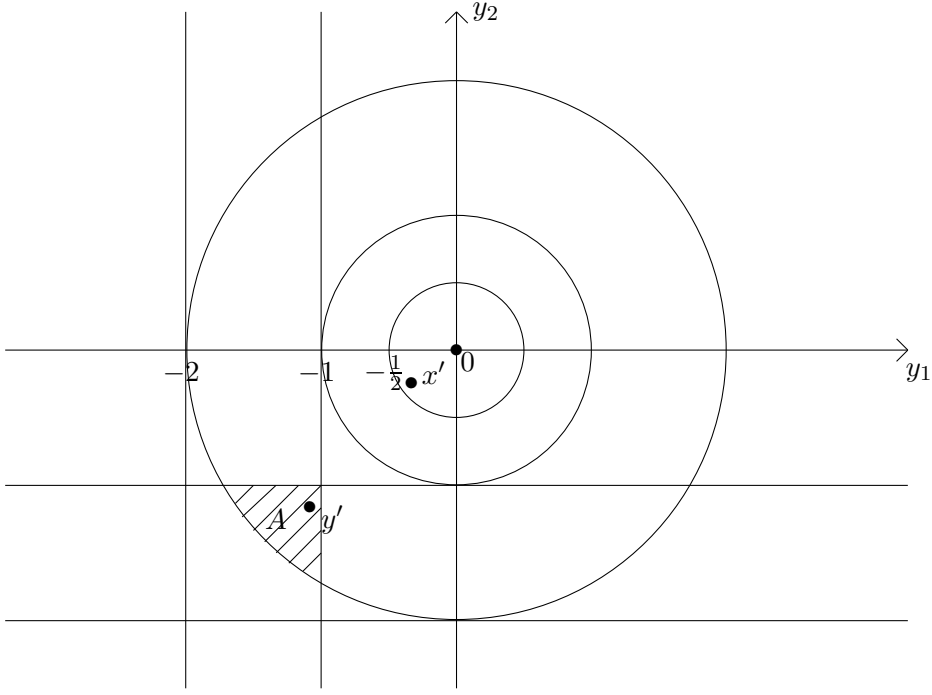
\draw (0.0, -1.7) node {Figure; A is the support of $g_n^1$ in $\mathbb R^2$};

\end{tikzpicture}


\begin{rem}\label{rem0101-1}
We note that
$g \in L^4(0,T;\dot{B}^{-\frac{1}{p}}_{pp}(\Rn))$ for $\frac{n}{n-1}  <p <
\infty$. In fact, for $ 1 < r <\infty$ satisfying $ -\frac{n-1}r = -\frac1p -\frac{n-1}p$, we can see that (refer to e.g. \cite[Theorem 6.5.1]{BL})
\begin{align*}
\| g(t)\|_{ \dot{B}^{-\frac{1}{p}}_{pp}(\Rn) } \leq c \| g(t)\|_{L^r(\Rn)} < \infty.
\end{align*}
\end{rem}

\begin{prop}\label{lemma0406}
Let $w$ be a solution of the Stokes equations
\eqref{Stokes-bvp-200}-\eqref{Stokes-bvp-210} with $F=0$ defined by
\eqref{rep-bvp-stokes-w}, where the boundary data $g$ is given in
\eqref{boundarydata} and \eqref{0502-6}. Let $ \frac{n}{n-1} < p < \infty$, $t_0 = \frac1{16}$ and
$r \leq \frac12$. Then, $w$ satisfies
\begin{equation}\label{L4Lp-w}
\| w\|_{L^4(0,\infty;L^p(\R_+))}<\infty,
\end{equation}
\begin{equation}\label{L2-grad-w}
\int_{t_0}^{ t_0 + r^2} \int_{B(0, r)} |D_{x_n} w_i(x,t) |^2 dxdt =
\infty \qquad  i =1,\cdots, n-1,
\end{equation}
\begin{equation}\label{L2-grad-w-n}
\int_{t_0}^{ t_0 + r^2} \int_{B(0, r)} |D_{x} w_n(x,t) |^2 dxdt <\infty.
\end{equation}
\end{prop}

\begin{proof}
We recall  Proposition \ref{thm-stokes} and Remark \ref{rem0101-1}, which implies that for $\frac{n}{n-1} <
p < \infty$
\begin{align*}
\| w\|_{L^4(0,\infty;L^p(\R_+))}\leq c_p \| g\|_{L^4(0,\infty;\dot{B}^{-\frac{1}{p}}_{pp}(\Rn))}  \leq c \| g\|_{L^4(0, \infty;L^r(\Rn))} < \infty.
\end{align*}
Hence, the first inequality \eqref{L4Lp-w} is immediate.

Next, we show the estimate \eqref{L2-grad-w}.  Since the arguments are similar,
we only prove the case of $i =1$. It follows from \eqref{Poisson-tensor-K} and
\eqref{rep-bvp-stokes-w} that
\begin{align*}
w_1(x,t) & =    \int_0^t   \int_{A} L_{1 n}(x'- y', x_n, t
-s) g_n(y', s) dy' ds -   \int_{A} D_{x_1} N(x' -y', x_n)g_n(y',t) dy'\\
&:=  w^1_{1}(x,t) + w^2_{1}(x,t).
\end{align*}
We note that $|x' -y'| \geq \frac12$ for $|x'| \leq \frac12  $ and $
y' \in A$, and so, for $(x, t) \in B(0, \frac12) \times (t_0,
\frac14) $ we have
\begin{align}\label{0411-w}
|\na w_1^2(x,t)| \leq c (t-t_0)^{-\frac14} |\ln ( t -t_0)|^{-\frac14 -\ep}.
\end{align}
Since $L_{1n} = L_{n1} + B_{1n}$ by the second equality of \eqref{1006-3}, we divide $w_1^1= w_1^{11} + w_1^{12}$ by
\begin{align}\label{0411-v}
\notag w^1_1(x,t) &=  \int_0^t \int_{\Rn} L_{n1}(x'- y', x_n, t -s) g_n(y',
s)dy'ds\\
\notag & \qquad  +\int_0^t \int_{\Rn} B_{1n}(x'- y', x_n, t -s) g_n(y', s)dy'ds\\
&: =
w^{11}_1 (x,t) + w^{12}_1 (x,t).
\end{align}

Now, we estimate of  $w^{11}_1$.  We assume that $ t -t_0 < x_n^2$. Note that $|x' -y'| \geq \frac12$ for $x \in B(0, \frac12) $ and $y' \in A$. Due to \eqref{est-L-tensor},
for $(x,t) \in B(0, \frac12) \times (t_0, \frac14)$, we have
\begin{align}
\notag |D_x w^{11}_1 (x,t)| 
 & \leq   c\int_{t_0}^t \int_{1 \leq |y'| \leq 2 } \frac1{ (t
-s)^\frac12 ( x_n^2 + t -s)^\frac12  ( |x' -y'|^2 + x_n^2 + t
-s)^\frac{n}2 }\\
\notag &\qquad \qquad \qquad \qquad \times  (s -t_0)^{-\frac14} |\ln (s- t_0)|^{-\frac14 -\ep} dy'
ds\\
\notag  &\leq   c
  x_n^{-1} \int_{t_0}^{t}  \frac1{ (t
-s)^\frac12  }
 (s-  t_0)^{-\frac14} |\ln (s- t_0)|^{-\frac14 -\ep} ds\\
\notag & = c x_n^{-1} \int_{t_0}^{\frac{t_0 +t}2}   \frac1{ (t
-s)^\frac12  }
 (s-  t_0)^{-\frac14} |\ln (s- t_0)|^{-\frac14 -\ep} ds\\
\notag&  \qquad  +  cx_n^{-1} \int^{t}_{\frac{t_0 +t}2}   \frac1{ (t
-s)^\frac12  }
 (s-  t_0)^{-\frac14} |\ln (s- t_0)|^{-\frac14 -\ep} ds\\
\label{w11-0}
 & = I_1 + I_2.
\end{align}
Since $t - t_0 < x_n^2$, we have
\begin{align} \label{0409I2}
\notag I_2 & \leq c (t-  t_0)^{-\frac34} |\ln (t- t_0)|^{-\frac14 -\ep} \int^{t}_{\frac{t_0 +t}2}   \frac1{ (t
-s)^\frac12  }   ds\\
 &  =
c   (t-  t_0)^{-\frac14} |\ln (t- t_0)|^{-\frac14 -\ep}.
\end{align}
On the other hand, noting that for $0 < a < \frac12$, we observe
\begin{align}\label{0409-1}
 c_1  a^{\frac34} |\ln a|^{-\frac14 -\ep} \leq   \int_0^a s^{-\frac14 } |\ln s|^{-\frac14 -\ep} ds \leq c_2  a^{\frac34} |\ln a|^{-\frac14 -\ep}.
\end{align}
Indeed, via Hospital's Theorem, we have
\[
\lim_{a \ri 0} \frac{ \int_0^a s^{-\frac14 } |\ln s|^{-\frac14 -\ep} ds }{ a^{\frac34} |\ln a|^{-\frac14 -\ep}} = \lim_{a \ri 0} \frac{ a^{-\frac14 } |\ln a|^{-\frac14 -\ep}  }{ \frac34 a^{-\frac14} |\ln a|^{-\frac14 -\ep} -(\frac14 +\ep)  a^{-\frac14} |\ln a|^{-\frac54 -\ep}} = \frac43.
\]
Since $a^{\frac34} |\ln a|^{-\frac14 -\ep}, \,    \int_0^a s^{-\frac14 } |\ln s|^{-\frac14 -\ep} ds > 0$ for $0 < a < \frac12$, which implies \eqref{0409-1}.

Due to \eqref{0409-1}, since $t -t_0 < x_n^2$, we obtain
\begin{align}
\notag I_1 & \leq  c (t -t_0)^{-1}  \int_{t_0}^{\frac{t_0 +t}2}
 (s-  t_0)^{-\frac14} |\ln (s- t_0)|^{-\frac14 -\ep} ds\\
\label{0409-2}
& \leq   c    (t-t_0 )^{-1}  \int_0^{\frac{t -t_0 }2}
 s^{-\frac14} |\ln s|^{-\frac14 -\ep} ds
 \leq    c   (t-  t_0)^{-\frac14} |\ln (t- t_0)|^{-\frac14 -\ep}.
\end{align}
Summing up \eqref{w11-0}, \eqref{0409I2} and \eqref{0409-2},
for $(x,t) \in B(0, \frac12) \times (t_0, \frac14)$ and $ t -t_0 < x_n^2$,  we have
\begin{align}\label{w11-1}
|D_x w^{11}_1 (x,t)| \leq  c   (t-  t_0)^{-\frac14} |\ln (t- t_0)|^{-\frac14 -\ep}.
\end{align}

Next, we estimate $w_1^{12}$. Reminding \eqref{B-tensor}, we note that
\begin{align}\label{w12}
\notag D_{x_n} w^{12}_1(x,t) &= c_n \int_{t_0}^t \int_{\Rn} g_n(y', s)  (t-s)^{-\frac{n+2}2}    \bke{-2 + \frac{4 x_n^2}{t-s}}\\
 & \qquad \qquad \times  e^{-\frac{x_n^2}{t-s}}
\int_{\Rn}  e^{-\frac{|x' -y'-z'|^2}{t-s}}   \frac{z_1}{|z'|^n}   dz' dy' ds.
\end{align}
For fixed $X' = x' -y'$, we divide $\Rn$ by three disjoint sets
$D_1, D_2$ and $D_3$ defined by
\[
D_1=\bket{z'\in\Rn: |X'-z'| \leq \frac1{10} |X'|},
\]
\[
D_2=\bket{z'\in\Rn: |z'| \leq \frac1{10} |X'|}, \qquad D_3=\Rn\setminus
(D_1\cup D_2).
\]
We then split the following integral into three terms as follows:
\begin{align*}
\int_{\Rn} e^{-\frac{|X'-z'|^2}{t-s}} \frac{z_1}{|z'|^n} dz' =
\int_{D_1}\cdots + \int_{D_2} \cdots+ \int_{D_3}\cdots := J_1 + J_2
+J_3.
\end{align*}
Since $\int_{D_2}\frac{z_1}{|z'|^n} dz' =0$,  we have $\int_{D_2} \frac{z_1}{|z'|^n} e^{-\frac{|X'-z'|^2}{t-s}}   dz' = \int_{D_2} \frac{z_1}{|z'|^n} \big(
e^{-\frac{|X'-z'|^2}{t-s}} -  e^{-\frac{|X'|^2}{t-s}} \big) dz'$.  Using the Mean-value Theorem, we  have
\begin{align}\label{est-J2}
\notag |J_2| & =\abs{\int_{D_2}   \frac{z_1}{|z'|^n} \big(
e^{-\frac{|X'-z'|^2}{t-s}} -  e^{-\frac{|X'|^2}{t-s}} \big) dz'
}\leq   c(t-s)^{-1 }|X'| e^{-c\frac{|X'|^2}{t-s}}  \int_{D_2}
\frac{1}{|z'|^{n-2}}   dz'\\
&\leq  c (t-s)^{-1 }|X'|^2 e^{-c\frac{|X'|^2}{t-s}} \leq  c
e^{-c\frac{|X'|^2}{t-s}}.
\end{align}
Since $\int_{|z'| > a} e^{-|z'|^2} dz' \leq c_1 e^{-c_2 a^2}, \, a > 0$, we have
\begin{align}
\notag|J_3| &\leq    \frac{c}{|X'|^{n-1}}\int_{D_3} e^{-\frac{|z' -X'|^2}{t-s}}
dz' \leq   \frac{c}{|X'|^{n-1}} \int_{\{|z'-X'| \geq \frac1{10} |X'|\}} e^{-\frac{|z'-X'|^2}{t-s}} dz'\\
\notag & =   \frac{c(t-s)^{\frac{n-1}2}}{|X'|^{n-1}} \int_{\{|z'| \geq \frac1{10}\frac{ |X'|}{\sqrt{t-s}}\}} e^{- |z'|^2} dz'\\
\label{est-J3}
&\leq   \frac{c(t-s)^{\frac{n-1}{2}}}{|X'|^{n-1}}  e^{-c\frac{ |X'|^2}{t-s}}\leq   c
e^{-c \frac{|X'|^2}{t-s}}.
\end{align}
Since    $ \frac12 \leq |X'| \leq 2$, from \eqref{est-J2} and \eqref{est-J3}, we have
\begin{align}\label{0604}
|J_2|, \, |J_3| \leq   c e^{-c\frac1{ ( t-s)}}
\end{align}

Now, we estimate $J_1$. We note that for $|x'| < \frac12$ and $y'
\in A$ we see that
$X_1 = x_1 - y_1 \geq -\frac12 + 1 =\frac12$ and
$\frac15 |X'| \leq \frac15 (|x'| + | y'|) \leq \frac15 \cdot \frac52
\leq X_1$. Then, for $|X'-z'| \leq \frac1{10} |X'| $,  we have
\[
z_1 = z_1 -X_1 + X_1 \geq X_1 - |z_1 - X_1| \geq X_1 - |X' -z'|
\]
\[
\geq X_1 - \frac1{10} |X'| \geq \frac15|X'| - \frac1{10} |X'| = \frac1{10} |X'|.
\]
Therefore,  we obtain
\begin{align}\label{est-J1}
\notag J_1 & = \int_{D_1} e^{-\frac{|X'-z'|^2}{t-s}} \frac{z_1}{|z'|^n} dz'
\geq \frac{c}{|X'|^{n-1}} \int_{D_1} e^{-\frac{|X'-z'|^2}{t-s}}
dz'\\
\notag & = \frac{c(t-s)^{\frac{n-1}2}}{|X'|^{n-1}}
\int_{\{|z'| \leq \frac1{10}\frac{ |X'|}{\sqrt{t-s}}\}}  e^{-|z'|^2}
dz'\\
& \geq \frac{c(t-s)^{\frac{n-1}2}}{|X'|^{n-1}} \geq c(t-s)^{\frac{n-1}2} .
\end{align}
Combining \eqref{0604} and \eqref{est-J1}, we
obtain
\begin{equation}\label{est-J1J2J3}
\int_{\Rn} e^{-\frac{|x'-y'-z'|^2}{(t-s)}} \frac{z_1}{|z'|^n} dz'
\geq   c  \Big(   (t -s)^{\frac{n-1}2}  -
e^{-c\frac{1}{(t-s)}} \Big) \geq   c (t -s)^{\frac{n-1}2} \Big(1     -
e^{-c\frac{1}{(t-s)}} \Big).
\end{equation}
Noting that $ -2 +  \frac{4 x_n^2}{t-s} > 2$
for $t-t_0  < x_n^2$ and   $t_0 < s< t$,  it follows from \eqref{w12} and \eqref{est-J1J2J3} that
\begin{align*}
|D_{x_n} w^{12}_1(x,t)| &\geq c \int_{t_0}^t \int_{A } (s-t_0)^{-\frac14}|\ln ( s -t_0)|^{-\frac14 -\ep}  (t-s)^{-\frac32}
  e^{-\frac{x_n^2}{t-s}}
\big( 1 - e^{-c\frac1{t-s}}  \big) dy' ds\\
&\geq c \int_{t_0}^t   (s-t_0)^{-\frac14}  |\ln (s-t_0)|^{-\frac14 -\ep} (t-s)^{-\frac32}     e^{-\frac{x_n^2}{t-s}} ds -c.
\end{align*}
Splitting the above integral over intervals $(t_0, \frac12 (t +
t_0))$ and $(\frac12 (t + t_0), t)$, we first compute
\begin{align*}
&\int_{t_0}^{\frac12 (t + t_0) }  (s-t_0)^{-\frac14} |\ln (
s-t_0)|^{-\frac14 -\ep} (t -s)^{-\frac32} e^{-\frac{x_n^2}{t -s}}    ds\\
&
\geq  c (t-t_0)^{-\frac32} e^{-\frac{x_n^2}{t -t_0}}
\int_{t_0}^{\frac12 (t + t_0)}       (s-t_0)^{-\frac14} |\ln
(s-t_0)|^{-\frac14 -\ep}  ds\\
&
\geq  c (t-t_0)^{-\frac32} e^{-\frac{x_n^2}{t -t_0}}
\int_0^{\frac12 (t - t_0)}       s^{-\frac14} |\ln
s|^{-\frac14 -\ep}  ds\\
&
\geq  c(t-t_0)^{-\frac34}|\ln (t-t_0)|^{-\frac14 -\ep}
e^{-\frac{x_n^2}{t-t_0}}.
\end{align*}
For the third inequality, we use $\int_0^a   s^{-\frac14} |\ln
s|^{-\frac14 -\ep}  ds \geq c   a^{\frac34} |\ln
a|^{-\frac14 -\ep}  $ for $ a \leq \frac12$ (see \eqref{0409-1}).

On the other hand, since $\int_a^\infty s^\frac12 e^{-s} ds \geq c
a^\frac12 e^{- a}$ for $ a > 1$, we have
\begin{align*}
&\int_{ \frac12 ( t + t_0)}^t (s-t_0)^{-\frac14} |\ln
(s-t_0)|^{-\frac14 -\ep} (t -s)^{-\frac32} e^{-\frac{x_n^2}{t -s}}
   ds\\
&\geq  c(t-t_0)^{-\frac14}|\ln (t-t_0)|^{-\frac14 -\ep} \int_{\frac12
(t+t_0)}^t (t -s)^{-\frac32} e^{-\frac{x_n^2}{t -s}}       ds\\
&=   c(t-t_0)^{-\frac14}|\ln (t-t_0)|^{-\frac14 -\ep} \int_0^{\frac12
( t-t_0)} s^{-\frac32} e^{-\frac{x_n^2}s}        ds\\
&\geq  c(t-t_0)^{-\frac14}|\ln (t-t_0)|^{-\frac14 -\ep} x_n^{-1}
\int^\infty_{\frac{2x_n^2}{t-t_0}} s^{\frac12} e^{-s}         ds\\
&\geq c(t-t_0)^{-\frac34} |\ln (t-t_0)|^{-\frac14 -\ep}
e^{-c\frac{x_n^2}{t-t_0}}.
\end{align*}
Adding the estimates above,  for $(x,t) \in B (0,\frac12) \times (t_0, \frac12)$ and  $t -t_0 \leq x_n^2$,  we have
\begin{align}\label{0411-v-2}
|D_{x_n} w^{12}_1(x,t)| &\geq c (t-t_0)^{-\frac34} |\ln (t-t_0)|^{-\frac14 -\ep} e^{-\frac{x_n^2}{t-t_0}} -c.
\end{align}
From \eqref{0411-w}, \eqref{w11-1} and \eqref{0411-v-2},
for $t -t_0 \leq x_n^2$ and $(x,t) \in B(0,\frac12) \times (t_0, \frac12)$, we get
\begin{equation}\label{grad-w1}
|D_{x_n} w_1(x,t)|
  \geq  c(t-t_0)^{-\frac34}|\ln (t-t_0)|^{-\frac14 -\ep}  e^{-\frac{x_n^2}{t-t_0}}
  -c (t-t_0)^{-\frac14} |\ln ( t -t_0)|^{-\frac14 -\ep}.
\end{equation}
Since
\begin{align*}
\notag & \int_{t_0}^\frac14 \int_{   \{(t-t_0)^\frac12 \leq x_n \leq
1 \}}  (t-t_0)^{-\frac12 } |\ln (t-t_0)|^{-\frac12 -2\ep}  dx_n dt <
\infty,
\end{align*}
from  \eqref{grad-w1}, for any $\ep$ with $0<\ep \leq  \frac14$,  we obtain
\begin{align}
\notag & \int_{t_0}^{t_0 + r^2} \int_{B_r(0)} |D_{x_n} w_1(x,t)|^2 dxdt \\
\notag
&\geq c\int_{t_0}^{t_0 + r^2}  \int_{\sqrt{t-t_0}}^r
\int_{|x'| < r}  (t-t_0)^{-\frac32} |\ln (t-t_0)|^{-\frac12 -2\ep}
e^{-\frac{x_n^2}{t-t_0}}  dx'dx_n dt-c\\
\notag
&\geq cr^{n-1} \int_0^{r^2}  \int_1^{\frac{r}{\sqrt{t}}}
t^{-1 } |\ln t|^{-\frac12 -2\ep}  e^{-x_n^2} dx_n dt-c\\
\label{L2-blowup}
&\geq cr^{n-1} \int_0^{\frac14 r^2}      t^{-1 } |\ln t|^{-\frac12
-2\ep} dt -c= \infty.
\end{align}
Therefore, we complete the proof of \eqref{L2-grad-w}.

It remains to prove \eqref{L2-grad-w-n}. It follows from \eqref{Poisson-tensor-K} and
\eqref{rep-bvp-stokes-w} that
\begin{align}\label{0212-1}
\notag w_n(x,t) & =   \int_0^t \int_{A} D_{x_n} \Ga(x'- y', x_n, t
-s)  g_n( y', s)  dy' ds\\
\notag  & \qquad    +    \int_0^t   \int_{A} L_{n n}(x'- y', x_n, t
-s) g_n( y', s) dy' ds
  -    \int_{A} D_{x_n} N(x' -y', x_n) g_n( y', t) dy'\\
& :=  w^1_n(x,t) + w^2_n(x,t ) + w^3_n (x,t).
\end{align}
As the same reason with \eqref{0411-w},  we obtain
\begin{align}\label{0502-5}
\int_{t_0}^{t_0 + r^2} |D_x w^3_n (x,t)|^2 dxdt < \infty.
\end{align}
We note that  by the first equality  \eqref{1006-3}, the kernel of $w^2_n$ also satisfies  \eqref{est-L-tensor}. With the same estimate \eqref{w11-1}, for $(x,t) \in B(0, \frac12) \times (t_0, \frac14)$ and $ t -t_0 < x_n^2$, we have
\begin{align}\label{wn1-1}
 |D_x w^2_n (x,t)|\leq   c   (t-  t_0)^{-\frac14} |\ln (t- t_0)|^{-\frac14 -\ep}.
\end{align}

Now, we assume that $ x_n^2 < t -t_0  $. From \eqref{est-L-tensor},
for $(x,t) \in B(0, \frac12) \times (t_0, \frac14)$, we have
\begin{align}
\notag |D_x w_n^2 (x,t)|
&\leq   c\int_{t_0}^t  \frac1{ (t
-s)^\frac12 ( x_n^2 + t -s)^\frac12  }  (s -t_0)^{-\frac14} |\ln (s- t_0)|^{-\frac14 -\ep}ds\\
\notag&\leq   c\int_{t_0}^{t-x_n^2}  \frac1{ (t
-s)^\frac12 ( x_n^2 + t -s)^\frac12  }  (s -t_0)^{-\frac14} |\ln (s- t_0)|^{-\frac14 -\ep}ds\\
\notag&\qquad +  c \int_{t -x_n^2}^t \frac1{ (t
-s)^\frac12 ( x_n^2 + t -s)^\frac12  }  (s -t_0)^{-\frac14} |\ln (s- t_0)|^{-\frac14 -\ep} ds\\
\label{w11-2}
& = II_1 + II_2.
\end{align}
Let  $2 x_n^2 < t-t_0$ ($\frac{t+t_0}2 < t - x_n^2$). Note that for $t -x_n^2 < s<t$, we have  $(s -t_0)^{-\frac14} |\ln (s- t_0)|^{-\frac14 -\ep} \leq c (t -t_0)^{-\frac14} |\ln (t- t_0)|^{-\frac14 -\ep}$. Hence, we have
\begin{align}\label{0501-2}
\notag II_2 &\leq  cx_n^{-1}  (t -t_0)^{-\frac14} |\ln (t- t_0)|^{-\frac14 -\ep}\int_{t -x_n^2}^t  \frac1{ (t
-s)^\frac12   }  ds\\
& = c   (t -t_0)^{-\frac14} |\ln (t- t_0)|^{-\frac14 -\ep}.
\end{align}
Since $\frac{t+t_0}2 < t - x_n^2$, from \eqref{0409-1}, we have
\begin{align} \label{0409I2-2}
\notag II_1 & \leq c (t -t_0)^{-1} \int_{t_0}^{\frac{t +t_0}2}      (s -t_0)^{-\frac14} |\ln (s- t_0)|^{-\frac14 -\ep}ds\\
 \notag & \qquad  + c (t -t_0)^{-\frac14} |\ln (t- t_0)|^{-\frac14 -\ep}\int_{\frac{t+t_0}2}^{t-x_n^2}   (t-s)^{-1}  ds\\
\notag   &  = c (t -t_0)^{-1} \int_{0}^{\frac{t -t_0}2}      s^{-\frac14} |\ln s|^{-\frac14 -\ep}ds
   + c (t -t_0)^{-\frac14} |\ln (t- t_0)|^{-\frac14 -\ep}\ln(\frac{2x_n^2}{t-t_0})\\
  &  \leq c (t -t_0)^{-\frac14  } |\ln (t-t_0)|^{-\frac14 -\ep}
   + c (t -t_0)^{-\frac14} |\ln (t- t_0)|^{-\frac14 -\ep}\ln(\frac{2x_n^2}{t-t_0}).
\end{align}
Hence, summing up \eqref{w11-2}, \eqref{0501-2} and \eqref{0409I2-2},
for $(x,t) \in B(0, \frac12) \times (t_0, \frac14)$ and $ 2x_n^2 <  t -t_0 $, we have
\begin{align}\label{0502w11-2}
|D_x  w_n^2 (x,t)| \leq  c   (t-  t_0)^{-\frac14} |\ln (t- t_0)|^{-\frac14 -\ep} \big( 1  + \ln(\frac{2x_n^2}{t-t_0}) \big).
\end{align}

Next, we assume that $ x_n^2 < t -t_0<2x_n^2  $ ( $\frac{t+t_0}2 > t - x_n^2> t_0$). From \eqref{0409-1}, we have
\begin{align}
\notag II_1
\notag &\leq   c (t -t_0)^{-1} \int_{t_0}^{t-x_n^2}   (s -t_0)^{-\frac14} |\ln (s- t_0)|^{-\frac14 -\ep}ds\\
\notag & \leq   c (t -t_0)^{-1} \int_0^{t-t_0}   s^{-\frac14} |\ln s|^{-\frac14 -\ep}ds\\
& \leq  c   (t-  t_0)^{-\frac14} |\ln (t- t_0)|^{-\frac14 -\ep},
\label{w11-3}
\end{align}
\begin{align} \label{0501-3}
\notag II_2&\leq c  (t -t_0)^{-1}  \int_{t -x_n^2}^{\frac{t +t_0}2}      (s -t_0)^{-\frac14} |\ln (s- t_0)|^{-\frac14 -\ep}ds\\
 &\qquad   + c  (t -t_0)^{-\frac34} |\ln (t- t_0)|^{-\frac14 -\ep} \int_{\frac{t+t_0}2}^{t} \frac1{ (t
-s)^\frac12  }   ds\\
\notag    &\leq  c    (t-  t_0)^{-\frac14} |\ln (t- t_0)|^{-\frac14 -\ep}.
\end{align}

From the  estimates \eqref{w11-2},  \eqref{w11-3} and  \eqref{0501-3},  for $(x,t) \in B_{\frac12} \times (t_0, \frac14)$ and $ x_n^2 < t -t_0<2x_n^2  $
\begin{align}\label{0502-2}
|D_x  w_n^2 (x,t)| \leq  c   (t-  t_0)^{-\frac14} |\ln (t- t_0)|^{-\frac14 -\ep}.
\end{align}

From \eqref{wn1-1}, \eqref{0502w11-2} and \eqref{0502-2}, we have
\begin{align}\label{0502-3}
\notag \int_{t_0}^{t_0 + r^2} \int_{B_r} | \na w^2_n (x,t)|^2 dxdt
 &  \leq c\int_{t_0}^{t_0 + r^2} \int_0^{   \sqrt{\frac{t -t_0}2}}    (t -t_0)^{-\frac12} |\ln (t- t_0)|^{-\frac12 -2\ep}\ln^2(\frac{2x_n^2}{t-t_0}) dx_ndt\\
&\qquad  + \int_{t_0}^{t_0 + r^2} \int_0^r (t-  t_0)^{-\frac12} |\ln (t- t_0)|^{-\frac12 -2\ep} dx_ndt
 < \infty.
\end{align}
Similarly , we have
\begin{align}\label{0502-4}
 \int_{t_0}^{t_0 + r^2} \int_{B_r} | \na w^1_n (x,t)|^2 dxdt
 & < \infty.
\end{align}

Hence,  from \eqref{0212-1}, \eqref{0502-5}, \eqref{0502-3} and \eqref{0502-4}, we complete the proof of \eqref{L2-grad-w-n}, and thus we deduce the proposition.
\end{proof}

\begin{rem}
From the estimates of the proof of Proposition \ref{lemma0406} (in particular $w^{12}_1$), for $(x,t) \in B_{\frac12} \times (t_0, \frac14)$, we can  obtain
\begin{align*}
| D_{x_n} w_i(x,t)| & \leq c (t-t_0)^{-\frac34} |\ln (t-t_0)|^{-\frac14 -\ep} e^{-\frac{x_n^2}{t-t_0}}\\
 &\qquad      + c(t -t_0)^{-\frac14} |\ln (t- t_0)|^{-\frac14 -\ep}\ln(\frac{2x_n^2}{t-t_0}) \chi_{x_n^2\leq t -t_0} \quad i = 1, \cdots, n-1
\end{align*}
and which implies  $D_{x_n} w_i\in L^4(t_0, \frac14; L^p( B_\frac12))$ for any $1\le p<2$.
\end{rem}


\section{Proof of Theorem \ref{maintheorem} }
\label{proofStokes}
\setcounter{equation}{0}

\subsection{Stokes equations}
We take a sequence  $
\{ g_{n,k}| k=1,2,\cdots\}$ with $ g_{n,k}\in C^\infty_c (A \times
(t_0, \frac14)) $ such that $g_{n,k}$ goes to $ g_n $ in $L^4(0,
\infty; L^p(\Rn) )$ as $k \ri \infty$, where $g_n$ is defined in \eqref{0502-6}. Let $\tilde{g}_k =(0,
g_{nk})$ and $w_k$ be solution of the Stokes equations
\eqref{Stokes-bvp-200}-\eqref{Stokes-bvp-210} defined by
\eqref{rep-bvp-stokes-w}. Since the boundary data $\tilde g_{k}$ are functions in $C^\infty_c(\Rn \times (0, \infty))$, we can obtain that $w_k $ is a classical
function, e.g.  $w_k \in \dot C^{l, \frac{l }2}(\overline \R_+
\times [0, \infty))\cap \dot W_p^{l, \frac{l}{2}} (\R_+\times [0, \infty))$
for any $l \in {\mathbb N}$. Note that $w_k =0$ on $B_\frac12 \times (t_0,\frac14) $.

Suppose that the Caccioppoli's inequality holds for smooth solutions
of the Stokes equations, that is,  $w_k$ satisfies the following
inequality for $k \in {\mathbb N}$;
\begin{equation}\label{h-cacciopoli}
\int_{t_0}^{t_0 + \frac14 r^2} \int_{B^+_{\frac12 r}(0)}  | \na  w_k(x,t)|^2 dxdt
\leq c r^{-2} \int_{t_0}^{t_0 + r^2} \int_{B^+_r(0)} | w_k(x,t)|^2 dxdt,
\end{equation}
where $c$ is independent of solutions.
Due to  Proposition \ref{thm-stokes}, we note that $w_k$ converges
to $u$ in $L^4(0, \infty; L^2 (\R_+ ))$ and on the other hand, by
the above inequality \eqref{h-cacciopoli}, $\na w_k$ converges to
$\na w$  in $L^2 (B^+_{\frac12 r}(0)  \times (t_0, t_0+ r^2))$. Hence, we
also get
\begin{align*}
\int_{t_0}^{t_0 + \frac14 r^2} \int_{B^+_{\frac12 r}(0)}  | \na  w(x,t)|^2 dxdt
\leq c r^{-2} \int_{t_0}^{t_0 + r^2} \int_{B^+_r(0)} | w(x,t)|^2 dxdt \leq C.
\end{align*}
This is, however, contrary to Proposition \ref{lemma0406}. Therefore, the
Caccioppoli's inequality is not true for the Stokes equations near
boundary. We complete the proof of the Theorem \ref{maintheorem} for
the case of Stokes equations.

\subsection{Navier-Stokes equations}
Let $g$ be a boundary data
defined in \eqref{0502-6} and $w$ be a solution of the Stokes
equations \eqref{Stokes-bvp-200}-\eqref{Stokes-bvp-210} defined by
\eqref{rep-bvp-stokes-w}. By the result of  Proposition \ref{lemma0406}, for $\frac{n}{n-1} <p < \infty$, we have
\begin{align}
\label{uc1-1}  \| w\|_{L^4(0,\infty;L^p (\R_+))} \leq c\| g\|_{L^4(0,\infty;\dot{B}^{-\frac{1}{p}}_{pp}(\Rn))} \leq c\al, \quad
 \| \na w\|_{L^2(B_r \times (t_0, t_0 + r^2))} = \infty
\end{align}
for all $ r>0$, where $\al >0$ is defined in \eqref{0502-6}.

Next, we consider the following perturbed Navier-Stokes equations in
$\R_+\times (0,\infty)$:
\begin{equation}\label{CCK-Feb7-10}
v_t-\Delta v+\nabla q+{\rm div}\,\left(v\otimes v+v\otimes
w+w\otimes v\right)=-{\rm div}\,(w\otimes w), \quad {\rm
div} \, v =0
\end{equation}
with homogeneous initial and boundary data, i.e.
\begin{equation}\label{pnse-bdata-20}
v(x,0)=0,\qquad v(x,t)=0 \,\,\mbox{on} \,\,\{x_n=0\}.
\end{equation}
Our aim is to establish the existence of solution  $v$ for
\eqref{CCK-Feb7-10} satisfying $v \in L^4(0, \infty; L^{2n}(\R_+))$
and $\na v \in L^2(0, \infty; L^{n}(\R_+))$. In order to do that, we
consider the iterative scheme for \eqref{CCK-Feb7-10}, which is
given as follows: For a positive integer $m\ge 1$
\begin{align*}
&v^{m+1}_t-\Delta v^{m+1}+\nabla q^{m+1}=-{\rm
div}\,\left(v^{m}\otimes v^{m}+v^{m}\otimes w+w\otimes
v^{m}+w\otimes w \right),\\
& \qquad \qquad \qquad \qquad \qquad {\rm div} \, v^{m+1} =0
\end{align*}
with homogeneous initial and boundary data, i.e. $v^{m+1}(x,0)=0$ and
$v^{m+1}(x,t)=0$ on $\{x_n=0\}$.
We set $v^1=0$. We then have, due to Proposition \ref{theo0503}, we have
\begin{align}\label{est-v2-10}
\notag \|\na  v^2\|_{L^2(0, \infty; L^{n}(\R_+))}  &\leq c \big(   \||w|^2\|_{L^2(0, \infty; L^{n}(\R_+)) } + \|w\otimes w|_{x_n =0}\|_{L^2(0, \infty; \dot B^{-\frac1n}_{nn} (\Rn)) } \big)\\
\notag & \leq c \big(  \|w\|^2_{L^4(0, \infty; L^{2n}(\R_+)) } + \|w\otimes w|_{x_n =0}\|_{L^2(0, \infty; L^{n-1} (\Rn)) }    \big)\\
& \leq c \big(  \|w\|^2_{L^4(0, \infty; L^{2n}(\R_+)) } + \|g\|^2_{L^4(0, \infty; L^{2(n-1)} (\Rn)) }    \big).
\end{align}
On the other hand, from Proposition \ref{thm-stokes},   we have
\begin{equation}\label{est-v2-20}
\| v^2\|_{L^4(0, \infty; L^{2n}(\R_+))}\leq c \|w \otimes
w\|_{L^2(0, \infty; L^{n}(\R_+))} \leq c   \|w\|^2_{L^4(0, \infty;
L^{2n}(\R_+)) }.
\end{equation}
By \eqref{uc1-1}, we have  $ A :=    \|w\|_{L^4(0, \infty; L^{2n}(\R_+)) } +  \|g\|_{L^4(0, \infty; L^{2(n-1)} (\Rn)) } \leq c\al$, where $\al>0$ is defined in \eqref{0502-6}. Taking $\al >0$ small such that
$A<\frac{1}{4c}$, where $c$ is the constant in
\eqref{est-v2-10}-\eqref{est-v2-20} such that
\begin{align*}
\|\na  v^2\|_{L^2(0, \infty; L^{n}(\R_+))}  +   \| v^2\|_{L^4(0, \infty; L^{2n}(\R_+))} < A.
\end{align*}
Then, iterative arguments show that
\begin{align}\label{0531-1}
\notag&\| \na v^{m+1}\|_{L^2(0, \infty; L^{n}(\R_+)) }\\
 \notag & \quad \leq c \big(
\||v^m|^2+|v^mw|+|w|^2\|_{L^2(0, \infty; L^{n}(\R_+))} + \|w\otimes w|_{x_n =0}\|_{L^2(0, \infty; \dot B^{-\frac1n}_{nn} (\Rn)) } \big) \\
\notag &\quad \leq  2c\left(\|v^m\|^2_{L^4(0, \infty; L^{2n}(\R_+))  }+\|w\|^2_{L^4(0, \infty; L^{2n}(\R_+))  } +  \|g\|^2_{L^4(0, \infty; L^{2(n-1)} (\Rn)) }\right)\\
 & \quad   \leq 4c A^2 <A.
\end{align}
Similarly, we note that
\begin{align}\label{0531}
\notag\| v^{m+1}\|_{L^4(0, \infty; L^{2n}(\R_+))} \leq c
\||v^m|^2+|v^mw|+|w|^2\|_{L^2(0, \infty; L^{n}(\R_+))}\\
\leq  2c\left(\|v^m\|^2_{L^4(0, \infty; L^{2n}(\R_+))  }+\|w\|^2_{L^4(0, \infty; L^{2n}(\R_+))  }\right)   \leq 4c A^2 <A.
\end{align}

We denote $V^{m+1}:=v^{m+1}-v^{m}$ and $Q^{m+1}:=q^{m+1}-q^{m}$ for $m\ge
1$. We then see that $(V^{m+1}, Q^{m+1})$ solves
\[
V^{m+1}_t-\Delta V^{m+1}+\nabla Q^{m+1}=-{\rm
div}\,\left(V^{m}\otimes v^{m}+v^{m-1}\otimes V^{m}+V^{m}\otimes
w+w\otimes V^{m}\right),
\]
\[
{\rm div} \, V^{m+1} =0,
\]
with homogeneous initial and boundary data, i.e. $V^{m+1}(x,0)=0$
and $V^{m+1}(x,t)=0$ on $\{x_n=0\}$. Taking sufficiently small
$\al>0$ such that $A < \frac1{6c}$, from \eqref{0531-1} and  \eqref{0531}, we obtain
\begin{align*}
\| \na V^{m+1}\|_{L^{2}(0, \infty; L^{n}(\R_+))}+\| V^{m+1}\|_{L^4 (0, \infty; L^{2n}(\R_+))}\\
\leq c \norm{|V^mv^m|+|V^mv^{m-1}|+|V^mw|}_{L^2(0, \infty; L^{n}(\R_+)) }\\
\leq 3c A \norm{V^m}_{L^4(0, \infty; L^{2n}(\R_+))}
<\frac12 \norm{V^m}_{L^4(0, \infty; L^{2n}(\R_+))}.
\end{align*}

Therefore, $(v^m, \na v^m)$ converges to $(v, \na v)$ in $L^4(0, \infty;
L^{2n}(\R_+))\times L^2(0, \infty; L^{n}(\R_+))$ such that $v$
solves in the sense of distributions
\[
v_t-\Delta v+\nabla \Pi=-{\rm div}\,\left(v\otimes v+v\otimes
w+v\otimes w+w\otimes w\right),
\]
\[
{\rm div} \, v =0,
\]
with homogeneous initial and boundary data, i.e. $v(x,0)=0$ and
$v(x,t)=0$ on $\{x_n=0\}$.

We then set $u:=v+w$ and $p =\pi + q$, which becomes a very weak solution of the
Navier-Stokes equations in $\R_+ \times (0, \infty)$, namely
\[
u_t-\Delta u+\nabla p=-{\rm div}\,\left(u\otimes u\right),\qquad
{\rm div} \, u=0,
\]
with boundary data $u(x,t)=g(x)$ on $\{x_n=0\}$ and homogeneous
initial data $u(x,0)=0$ such that
\begin{align}\label{0502-7}
 \| u \|_{L^4(0, \infty;  L^{2n}(\R_+))} \leq c,\qquad \| \na u\|_{L^2(B_r \times (t_0, t_0 + r^2))} = \infty \qquad \forall r >0.
\end{align}

Similarly, as in case of Stokes equations, we take  $  g_{n,k} \in
C^\infty_c (A \times (t_0, \frac14)) $  such that $g_{n,k}$ goes to
$g_n$ in $L^4(0,\infty; L^{2n}(\Rn ))$ as $k \ri \infty$. We denote
$\tilde{g}_k =(0,  g_{nk})$. Let $w_k$ be a solution of Stokes equations
with boundary data $\tilde{g}_k$ and $v_k$ be a solution of
\eqref{CCK-Feb7-10}-\eqref{pnse-bdata-20} with replacement of $w$ by
$w_k$. We recall that $w_k$ is smooth vector field such that $w_k $
converges to $w$ in $L^4 (0,\infty ; L^{2n}(\R_+ ))$.

We also observe that $v_k$ is smooth vector field such that
\[
\| \na v_k\|_{L^2(0, \infty;L^{n}(\R_+))}+\| v_k \|_{L^4(0,\infty;
L^{2n}(\R_+))} \leq c.
\]
By weak compactness, there is a subsequence of $\{ v_k\}$, redefined
as $v_k$, and $v\in L^4(0, \infty;L^{2n}(\R_+))$ such that $v_k$
weakly converges in $L^4(0, \infty;L^{2n}(\R_+))$ to $v$.   This
implies that $u_k = w_k + v_k$ weakly converges to $u = w + v$ in
$L^4(0,\infty; L^{2n}(\R_+))$.

Suppose that the Caccioppoli's inequality holds for smooth solutions
for the Navier-Stokes equations, that is, for $k \in {\mathbb N}$,
$u_k$ is assumed to satisfy the following inequality;
\begin{align*}
\int_{t_0}^{t_0 + \frac14 r^2} \int_{B^+_{\frac12 r}(0)}  | \na  u_k(x,t)|^2 dxdt
&\leq c r^{-2} \int_{t_0}^{t_0 + r^2} \int_{B^+_r(0)} | u_k(x,t)|^2 dxdt\\
&\qquad  + c r^{-1} \int_{t_0}^{t_0 + r^2} \int_{B^+_r(0)} | u_k(x,t)|^3 dxdt,
\end{align*}
where $c>0$ is independent of $k$. Since $\| u_k \|_{L^4(0,
\infty;L^{2n}(\R_+))} \leq c$ for all $k$, this leads to
\begin{align*}
\int_{t_0}^{t_0 + \frac14 r^2} \int_{B^+_{\frac12 r}(0)}  | \na  u(x,t)|^2 dxdt
\leq c <\infty.
\end{align*}
This is, however, contrary to \eqref{0502-7}, and therefore,
the Caccioppoli's inequality is not true for the Navier-Stokes
equations near boundary. This completes the proof of Theorem \ref{maintheorem}. \qed

\appendix
\setcounter{equation}{0}

\section{Proof of Proposition \ref{theo0503}}
\label{appendixa}

\subsection{Helmholtz projection  in half space} \label{projection}
\setcounter{equation}{0}

It is well known that the Helmholtz projection   ${\mathbb P}$ in half space $\R_+$ is given  by
\begin{align}\label{Hprojection}
{\mathbb P} f = f - \na {\mathbb Q}_1f - \na {\mathbb Q}_2 f = f -\na {\mathbb Q}f,
\end{align}
where ${\mathbb Q}_1 f$ and ${\mathbb Q}_2f $ satisfy the following equations;
\begin{align*}
\De {\mathbb Q}_1 f ={\rm div}\, f,\qquad
{\mathbb Q}_1 f|_{x_n =0} =0
\end{align*}
and
\begin{align*}
\De{\mathbb Q}_2f =0,\qquad
D_{x_n} {\mathbb Q}_2f|_{x_n =0} = \big(f_n -D_{x_n}{\mathbb Q}_1f\big)|_{x_n =0}.
\end{align*}
Note that ${\mathbb Q}_1f$ and ${\mathbb Q}_2f$ are represented by
\begin{align}\label{0427-1}
{\mathbb Q}_1f(x)& =- \int_{\R_+}  \na_y \big(N(x- y) - N (x - y^* ) \big) \cdot  f(y) dy,\\
\label{0427-2} {\mathbb Q}_2f (x) & = \int_{\Rn}   N(x'-y',x_n)  \big(f_n(y', 0)  - D_{y_n} {\mathbb Q}_1f (y', 0) \big) dy',
\end{align}
where $y^*= (y', -y_n)$.
Note that $({\mathbb P}f)_n|_{x_n =0} =0$.

\begin{lemm}\label{lemma0523-3}

Let  $f_i = {\rm div} \, F_i $ and   $f = ( f_1, \cdots, f_n)$.   Then,
\begin{align*}
{\mathbb Q}_1  f(x) &=\sum_{ k \neq n} D_{x_k} {\mathbb Q}_1 F_k(x) +  D_{x_n} A(x),\\
{\mathbb Q}_2 f(x) &= \sum_{  k \neq  n} D_{x_k}  {\mathbb Q}_2 F_k(x) - \sum_{  k \neq  n} D_{x_k} B^1_k(x)  - \sum_{ k \neq n} D^2_{x_k} B^2(x),\\
&\qquad ,
\end{align*}
where
\begin{align*}
A(x) & = -\int_{\R_+}\na_y \big( N(x-y) + N(x-y^*) \big)\cdot F_n(y) dy +2 B^1_n(x),\\
B^1_k(x) & =  -  \int_{\Rn} N(x' -y', x_n) F_{kn} (y', 0) dy', \qquad 1 \leq k \leq n,\\
B^2(x) & =  \int_{\Rn} N(x'-y',x_n) A(y',0) dy'.
\end{align*}

\end{lemm}

\begin{proof}

From \eqref{0427-1}, we have
\begin{align}\label{0424-1}
\notag {\mathbb Q}_1  f(x)
\notag &= - \sum_{ k  \neq n} D_{x_k}\int_{\R_+} \na_y \big( N(x-y) - N(x-y^*) \big) \cdot F_k(y) dy\\
\notag & \quad  -  D_{x_n} \int_{\R_+}\na_y \big( N(x-y) + N(x-y^*) \big)\cdot F_n(y) dy\\
\notag &\quad +2  D_{x_n} \int_{\Rn} N(x'-y', x_n) F_{nn} (y', 0) dy'\\
&:=\sum_{ k \neq n} D_{x_k} {\mathbb Q}_1 F_k(x) +  D_{x_n} A(x).
\end{align}

Since $\De A = {\rm div} \, F_n = f_n$, from \eqref{0424-1}, we have  $D_{y_n}{\mathbb Q}_1f (y) = D_{y_n} \sum_{k \neq n} D_{y_k} {\mathbb Q}_1 F_k(y)  +  f_n (y) - \De' A(y) $. Hence, we have
\begin{align}\label{0501-1}
\notag \int_{\Rn} N(x' -y', x_n) D_{y_n} {\mathbb Q}_1f (y',0) dy'
\notag & = \sum_{  k \neq  n}  D_{x_k} \int_{\Rn} N(x'-y',x_n)   D_{y_n} {\mathbb Q}_1 F_k(y',0)  dy'\\
\notag &\quad  + \int_{\Rn} N(x'-y',x_n)f_n(y',0) dy'\\
&\quad  + \sum_{ k \neq n} D_{x_k} \int_{\Rn} N(x'-y',x_n) D_{y_k}A(y',0) dy'.
\end{align}

Hence, from \eqref{0427-2} and  \eqref{0501-1}, we have
\begin{align*} 
{\mathbb Q}_2 f(x) &=
-\sum_{  k \neq  n}  D_{x_k} \int_{\Rn} N(x'-y',x_n)  D_{y_n} {\mathbb Q}_1 F_k(y',0)  dy'\\
&\quad  - \sum_{ k \neq n} D_{x_k} \int_{\Rn} N(x'-y',x_n) D_{y_k}A(y',0) dy'.
\end{align*}
We complete the proof.
\end{proof}

\subsection{Proof of Proposition \ref{theo0503}}
To prove Proposition \ref{theo0503}, we use the following Proposition.
\begin{prop}(\cite[Proposition 3.2]{CJ})
\label{0503prop2}
 Let $1<p,q<\infty$. Let $\Ga *' g(x,t) =  \int_{-\infty}^t \int_{\Rn} \Ga(x' -y', x_n, t-\tau)g(y',\tau) dy' d\tau $.
Then
\begin{align*}
\| D_x \Ga *' g\|_{L^q({\mathbb R};L^p(\R_+))}  \leq c \|  g\|_{L^q({\mathbb R};\dot{B}^{-\frac{1}{p}}_{pp}(\Rn))}.
\end{align*}
\end{prop}

We consider the Stokes equations
\begin{align}\label{maineq-stokesh=0}
\begin{array}{l}\vspace{2mm}
v_t - \De v + \na \Pi =f, \qquad \mbox{div } v =0, \mbox{ in }
 \R_+\times (0,\infty),\\
\hspace{30mm}v|_{t=0}= 0, \qquad  v|_{x_n =0} = 0,
\end{array}
\end{align}
where $f = \mbox{div}\, F$.

Let  $f = {\mathbb P}f + \na {\mathbb Q} f$ be  a  decomposition of $f$ defined \eqref{Hprojection}. Note that $({\mathbb P} f)_n|_{x_n =0} =0$.  We define  $(v, \Pi_0)$  by
\begin{equation}\label{expression-v}
v_i (x,t) =\int_0^t \int_{{\mathbb R}^n_+} G_{ij}(x,y, t-\tau)
({\mathbb P} f)_j(y,\tau) dyd\tau,
\end{equation}
\begin{equation*}
\Pi_0(x,t) =\int_0^t \int_{{\mathbb R}^n_+} P(x,y, t-\tau)
\cdot  ({\mathbb P} f) (y,\tau)dyd\tau,
\end{equation*}
where $G$ and $P$ are defined by
\begin{align}\label{formulas-v}
G_{ij} &= \de_{ij} (\Ga(x-y, t) - \Ga(x-y^*,t )) + 4(1 -\de_{jn})
\frac{\pa}{\pa x_j} \int_0^{x_n} \int_{{\mathbb R}^{n-1}}
            \frac{\pa N(x-z)}{\pa x_i} \Ga(z -y^* , t) dz,
\end{align}
\begin{align*}
\notag P_j(x,y,t) & =4 (1 - \de_{jn}) \frac{\pa }{\pa x_j}\Big[ \int_{{\mathbb
R}^{n-1}} \frac{\pa N(x' - z', x_n)}{\pa x_n} \Ga(z' -y', y_n,t) dz'\\
& \quad +\int_{{\mathbb R}^{n-1}} N(x' -z',x_n) \frac{\pa \Ga(z'-y', y_n,
t)}{\pa y_n} dz'\Big].
\end{align*}
From \cite{So}, $(v, \Pi_0)$ satisfies
\begin{align*}
\begin{array}{l}\vspace{2mm}
v_t - \De v + \na \Pi_0 = {\mathbb P}\, f, \qquad \mbox{div } v =0, \mbox{ in }
 \R_+\times (0,\infty),\\
\hspace{30mm}v|_{t=0}= 0, \qquad  v|_{x_n =0} = 0.
\end{array}
\end{align*}
Let $\Pi = \Pi_0 + {\mathbb Q}f$. Then, $(v, \Pi)$ is solution of \eqref{maineq-stokesh=0}.

Let $1<p, \, q <\infty$.  In Section 3 in  \cite{CK}, the authors showed that
 $v $ defined by
\eqref{expression-v}   has the following estimate;
\begin{align}\label{0417-2}
\|  \na v \|_{L^q(0, \infty;  L^{p}(\R_+)) } &
   \leq c \big( \| \na   \Gamma* {\mathbb P} f\|_{ L^q(0, \infty;  L^{p}(\R_+)) } +
  \|  \na \Gamma^* * {\mathbb P} f \|_{ L^q(0, \infty; L^{p}(\R_+)} \big),
\end{align}
where $\Gamma^* *  f (x,t) = \int_0^t \int_{\R_+} \Ga(x-y^*, t-\tau) f(y, \tau) dyd\tau$.

\begin{lemm}\label{0929-1}
Let $1 <  p,\, q < \infty$.
Let  $F \in L^{q} (0, \infty, L^p(\R_+))$. Then,
\begin{align*}
\notag \| \na  \Ga* {\mathbb P}({\rm div} \, F)\|_{L^q (0, \infty; L^{p}(\R_+)} + \| \na \Ga^* * {\mathbb P}({\rm div} \, F)\|_{L^q(0, \infty;  L^{p}(\R_+))}\\
 \leq c \big(\| F\|_{L^{q} (0,\infty;  L^p (\R_+))} +  \| F|_{x_n =0}\|_{L^q (0, \infty; \dot B^{-\frac1p}_{pp}(\Rn))}\big).
\end{align*}

\end{lemm}
\begin{proof}
Since the proofs are exactly same, we only prove the case of $ \Ga^* * {\mathbb P}({\rm div} \, F)$.

Since $({\mathbb P}\, {\rm div}\, F)_j (t) = {\rm div}\, F_j(t) - D_{x_j} {\mathbb Q} \, {\rm div}\, F(t)$,   for $1 \leq j \leq n-1$, we have
\begin{align*}
\notag (\Ga^* * ({\mathbb P}\, {\rm div}\, F) )_j (x,t) &= \int_0^t \int_{{\mathbb R}^n_+}
 \na_y \Ga(x-y^*, t-\tau) \cdot F_{j}(y,\tau) dyd\tau \\
\notag &\quad  + \int_0^t \int_{\Rn}
   \Ga(x'-y', x_n, t-\tau) F_{jn}(y',0,\tau) dy' d\tau \\
&\quad +  \int_0^t \int_{{\mathbb R}^n_+}
D_{y_j} \Ga(x-y^*, t-\tau) {\mathbb Q}({\rm div} \, F )(y,\tau) dyd\tau.
\end{align*}
Note that $D_{x_n}^2 A(t) = -\De' A (t)+ {\rm div }\, F_n(t)$. From Lemma \ref{lemma0523-3}, we have
\begin{align*}
\notag \Ga^* * ({\mathbb P}\, {\rm div}\, F)_n (x,t)
&= \int_0^t \int_{{\mathbb R}^n_+}  \na_y \Ga(x-y^*, t-\tau) \cdot F_n(y,\tau) dyd\tau dz'\\
\notag &\quad  + \int_0^t \int_{\Rn}
   \Ga(x'-y', x_n, t-\tau)  F_{nn}(y',0,\tau) dy' d\tau \\
\notag &\quad + \sum_{1 \leq k \leq n-1} D_{x_k}  \int_0^t \int_{{\mathbb R}^n_+}  \Ga(x-y^*, t-\tau)D_{y_n} {\mathbb Q}  F_k(y,\tau) dyd\tau\\
\notag &\quad +  \int_0^t \int_{{\mathbb R}^n_+} \na' \Ga(x-y^*, t-\tau)\cdot \na'A(y,\tau) dyd\tau\\
\notag &\quad +  \int_0^t \int_{{\mathbb R}^n_+} \na' \Ga(x-y^*, t-\tau)\cdot \na' D_{y_n}B^1_k(y,\tau) dyd\tau\\
&\quad +  \int_0^t \int_{{\mathbb R}^n_+} \na' \Ga(x-y^*, t-\tau)\cdot \na' D_{y_n}B^2(y,\tau) dyd\tau.
\end{align*}

Due to parabolic type's Calderon-Zygmun Theorem and Proposition \ref{0503prop2},  for  $ 1 < p, q< \infty$, we have
\begin{align*}
\|\na \Ga^* * ({\mathbb P}\, {\rm div}\, F)\|_{L^q(0, \infty; L^p(\R_+)) }
& \leq   \big(
    \| D_x {\mathbb Q} F  \|_{ L^q(0,   \infty; L^p(\R_+))} + \|D_x A\|_{ L^q(0,   \infty; L^p(\R_+))}\\
     & \quad  +  \|D_x B^1_k F_k \|_{ L^q(0,   \infty; L^p(\R_+))} +   \|D^2_x B^2 \|_{ L^q(0,   \infty; L^p(\R_+))}\\
     & \quad + \| F\|_{L^q(0,   \infty; L^p(\R_+)) }    + \| F|_{x_n =0}\|_{L^q (0, \infty; \dot B^{-\frac1p}_{pp}(\Rn))}   \big).
\end{align*}
It is well known (see e.g. \cite{St})  that  $D B^1_k(t)$ are bounded from $\dot{B}^{l-\frac{1}{p}}_{pp}(\Rn)$ to $\dot W^l_p(\R_+)$, $l \in {\mathbb N} \cup \{ 0\}$ so that
\begin{align}\label{Poisson}
  \| D_x B^1_k(t)\|_{\dot W^l_p(\R_+)}\leq c\|F(t)|_{x_n =0}\|_{\dot B^{l-\frac{1}{p}}_{pp}(\Rn)}.
 \end{align}

By Calderon-Gygmund Theorem and \eqref{Poisson}, we have
\begin{align*}
\| D_x {\mathbb Q}_1 F(t)\|_{L^p (\R_+)} & \leq c \|  F(t)\|_{ L^p (\R_+)},\\
 \| D_x A(t) \|_{ L^p(\R_+)}  & \leq c \|  F(t)\|_{ L^p(\R_+)} + c\|F(t)|_{x_n =0}\|_{\dot B^{k-\frac{1}{p}}_{pp}(\Rn)},\\
\|D^2_x B^2(t) \|_{L^p (\R_+)} & \leq  c \| A(t)|_{x_n =0} \|_{\dot B^{1  -\frac1p}_{pp} (\Rn)} \leq  c \| D_xA(t) \|_{L^p (\R_+)}\\
 &\leq c \| F(t)\|_{L^p (\R_+)} + c\|F(t)|_{x_n =0}\|_{\dot B^{-\frac{1}{p}}_{pp}(\Rn)}.
\end{align*}
Finally, from \eqref{Poisson}, we have
  \begin{align*}
 \| D_x  {\mathbb Q}_2 F_k (t)\|_{L^p (\R_+)} & \leq c  \|\big( F_k (t)- \na {\mathbb Q}_1F_k(t) \big)_n|_{x_n =0} \|_{\dot B^{-\frac1p}_{pp} (\Rn)}.
  \end{align*}
Since $F_k(t) - \na {\mathbb Q}_1F_k(t)$ is in $ \in L^p (\R_+)$ and  divergence free in $\R_+$, its normal component  has trace  $ \big( F_k(t) - \na {\mathbb Q}_1F_k (t) \big)_n|_{x_n =0} \in \dot B^{ -\frac1p}_{pp} (\Rn)$  (see \cite{galdi}). Hence,  we have
 \begin{align*}
\|\big( F_k(t) - \na {\mathbb Q}_1F_k (t) \big)_n|_{x_n =0} \|_{\dot B^{ -\frac1p}_{pp} (\Rn)} \leq  \|F_k(t) - \na {\mathbb Q}_1F_k (t) \|_{L^p (\R_+)}\leq c \|F_k(t)   \|_{ L^p (\R_+)}.
 \end{align*}
Therefore, we complete the proof of Lemma \ref{0929-1}.
\end{proof}
With the aid of the estimate \eqref{0417-2} and Lemma \ref{0929-1}, Proposition  \ref{theo0503} is immediate. 


\section*{Acknowledgements}
T. Chang is partially supported by NRF-2017R1D1A1B03033427 and K. Kang is partially
supported by NRF-2017R1A2B4006484 and NRF-2015R1A5A1009350.

\begin{equation*}
\left.
\begin{array}{cc}
{\mbox{Tongkeun Chang}}\qquad&\qquad {\mbox{Kyungkeun Kang}}\\
{\mbox{Department of Mathematics }}\qquad&\qquad
 {\mbox{Department of Mathematics}} \\
{\mbox{Yonsei University
}}\qquad&\qquad{\mbox{Yonsei University}}\\
{\mbox{Seoul, Republic of Korea}}\qquad&\qquad{\mbox{Seoul, Republic of Korea}}\\
{\mbox{chang7357@yonsei.ac.kr }}\qquad&\qquad
{\mbox{kkang@yonsei.ac.kr }}
\end{array}\right.
\end{equation*}

\end{document}